\documentclass[10pt,reqno]{amsart}

\usepackage{a4wide}

\usepackage{amsrefs}
\usepackage{mathrsfs}

\usepackage{amsmath}
\usepackage{amsthm}
\usepackage{amsfonts}
\usepackage{amssymb}
\usepackage[super]{nth}

\theoremstyle{plain}
\newtheorem{thm}{Theorem}[section]

\theoremstyle{definition}
\newtheorem{rem}[thm]{Remark}

\numberwithin{thm}{section}
\numberwithin{equation}{section}

\newcommand\ddfrac[2]{\frac{\displaystyle #1}{\displaystyle #2}}

\def\esup{\operatornamewithlimits{ess\,sup}}
\def\Id{\operatorname{I}}

\def\dual{\,^{^{\bf c}}\!}
\def\supp{\operatorname{supp}}

\allowdisplaybreaks

\begin{document}

\author{Amiran Gogatishvili \and Tu\u{g}\c{c}e \"{U}nver}

\title{Embeddings Between weighted local Morrey-type spaces}

\address{A. Gogatishvili,
Mathematical Institute, Czech Academy of Sciences,
\v Zitn\'a~25,
115~67 Praha~1, Czech Republic}
\email{gogatish@math.cas.cz}

\address{T. \"{U}nver, Department of Mathematics, Faculty of Science and Arts, Kirikkale University, 71450 Yahsihan, Kirikkale, Turkey}
\email{tugceunver@kku.edu.tr}



\thanks{
The research of  the first author  was partially supported by grant no. P201-18-00580S of the Grant Agency of the Czech Republic, by RVO: 67985840, and by Shota Rustaveli National Science Foundation (SRNSF), grant no: FR17-589.}

\dedicatory{Dedicated to the \nth{75} birthday of Professor Lars Erik Persson}

\begin{abstract}
In this paper certain $n$-dimensional inequalities are shown to be equivalent to the inequalities in the one-dimensional setting. By this means, embeddings between weighted local Morrey-type spaces are characterized for some ranges of parameters.  
\end{abstract}


\maketitle

\section{Introduction} \label{intro}

In this paper, we study some inequalities involving one-dimensional and $n$-dimensional Hardy-type 
operators. Our goal is to find the weight characterizations of inequality
\begin{align}\label{main inequality}
&\bigg(\int_0^{\infty} \bigg(\int_{B(0,t)} f(s)^{p_2} v_2(s)^{p_2} 
ds\bigg)^{\frac{q_2}{p_2}} w_2(t)^{q_2} dt\bigg)^{\frac{1}{q_2}} \nonumber 
\\
&\hspace{3cm}\leq c \bigg(\int_0^{\infty} \bigg(\int_{B(0,t)} f(s)^{p_1} 
v_1(s)^{p_1} ds\bigg)^{\frac{q_1}{p_1}} w_1(t)^{q_1} 
dt\bigg)^{\frac{1}{q_1}},
\end{align}
where $0 < p_1, p_2, q_1, q_2 < \infty$ and $v_1, v_2, w_1, w_2$ are non-negative measurable functions. In order to characterize \eqref{main inequality} we will use the solutions of the corresponding one-dimensional inequality. 

Let us first present some  notations used in this paper.

Let $A$ be a nonempty measurable subset of $\mathbb{R}^n, \, n\ge 1$. By  $\mathcal M(A)$, we 
denote the set of all measurable functions on $A$.  We denote by $\mathcal M^+(A)$, the set of all 
nonnegative measurable functions on $A$. A weight is a measurable, positive and finite a.e  
function on $A$ and we will denote the set of weights by $\mathcal W(A)$.

For $p\in [0, \infty)$, we define the functional $\|\cdot\|_{p,A}$ on $\mathcal{M}(A)$ by
\begin{equation*}
\|f\|_{p,A}:= \bigg\{\begin{array}{ccc}
\big(\int_A |f(x)|^p dx\big)^{\frac{1}{p}}& if  & p< \infty,  \\ 
\esup\limits_{x\in A} |f(x)|& if & p=\infty.
\end{array} \bigg.
\end{equation*}
If $w\in \mathcal{W}(A)$, then the weighted Lebesgue space $L_{p,w}(A)$ is given by
\begin{equation*}
L_{p,w}(A) := \{f \in \mathcal{M}(A): \|fw\|_{p,A}< \infty\}.
\end{equation*}

If $A=(a,b), \, 0\le a<b\le\infty$,   then we write  $\mathcal M(a,b)$,  $\mathcal M^+(a,b)$, $\mathcal W^+(a,b)$, $L_p(a,b)$ and $L_\infty(a,b)$, correspondingly.

For $x \in \mathbb{R}^n$ and $r >0$, let $B(x,r)$ be the open ball centered at $x$ of radius $r$ 
and $\dual B(x,r) = \mathbb{R}^n \backslash B(x,r)$.

Let $X$ and $Y$ be quasi normed vector spaces. If $X \subset Y$ and the identity operator is continuous from $X$ to 
$Y$, that is, there exists a positive constant $C$ such that $\|\Id(z)\|_Y \leq C \|z\|_X$ for all $z \in X$, we say 
that $X$ is embedded into $Y$ and write $X \hookrightarrow Y$.  Throughout the paper, we always denote by $c$ and $C$ a 
positive constant which is independent of main parameters but it may vary from line to line. However, a constant with 
subscript such as $c_1$ does not change in different occurrences. The standard notation $A \lesssim B $ means that 
there exists a constant $\lambda >0$ depending only on inessential parameters such that $A \leq \lambda B$ and we write 
$A \approx B$ if both $A \lesssim B$ and $B \lesssim A$. Since the expressions
in our main results are too long, to make the formulas simpler we sometimes omit the differential element $dx$. We will denote the left-hand side (right-hand side) of an inequality $*$ by $LHS(*)\,  (RHS(*))$.
 
We denote by  $LM_{p,q}(v, w)$, the weighted local Morrey-type spaces, the collection of all functions $f \in L_{p,v}^{loc}(\mathbb{R}^n)$ such that
\begin{equation*}
\|f\|_{LM_{p,q}(v, w)} = \bigg(\int_0^{\infty} \bigg( \int_{B(0,t)} f(s)^p v(s)^p ds \bigg)^{\frac{q}{p}} w(t)^q dt \bigg)^{\frac{1}{q}} < \infty,
\end{equation*}	
and $\dual LM_{p,q}(v, w)$, the weighted complementary local Morrey-type spaces, the collection of all functions $f \in L_{p, v}(\dual B(0, t))$ such that
\begin{equation*}
\|f\|_{\dual LM_{p,q}(v, w)} = \bigg(\int_0^{\infty} \bigg( \int_{\dual B(0,t)} f(s)^p v(s)^p ds \bigg)^{\frac{q}{p}} w(t)^q dt \bigg)^{\frac{1}{q}} < \infty,
\end{equation*}	
where $p,q \in (0, \infty)$, $w \in \mathcal{M}^+(0,\infty)$ and $v \in \mathcal{W}(\mathbb{R}^n)$

Under these notations, \eqref{main inequality} is equivalent to the continuous embedding between weighted local Morrey-type spaces, that is,
\begin{equation}\label{emb. LM1-LM2}
LM_{p_1, q_1}(v_1, w_1) \hookrightarrow LM_{p_2, q_2}(v_2, w_2).
\end{equation}

The properties of local Morrey-type spaces and boundedness of the classical operators in these spaces have been studied intensively (for the detailed history and related results of these spaces we  refer to the survey papers \cite{Burenkov-Survey-I} and \cite{Burenkov-Survey-II}).  Some necessary and sufficient conditions for the boundedness of the maximal operator, fractional maximal operator, Riesz potential and singular integral operator in the local Morrey-type spaces $LM_{p,q}(1, w) (\dual LM_{p,q}(1, w))$ were given in \cite{BurGul2004, BurGulGul2007-1, BurGulGul2007-2, BurGulTarSer, BurGul2009, BurGogGulMus}. 

The main tool in the literature to investigate the boundedness of the aforementioned operators in local Morrey-type spaces is to reduce these problems to the boundedness of the Hardy operator in weighted Lebesgue spaces on the cone of nonnegative, nonincreasing functions with the help of H\"{o}lder's inequality. But in \cite{BurGol} authors characterized the boundedness of maximal operator from $L_{p_1}$ to $LM_{p_2, q}(1, w)$ by using the characterization of the embeddings $L_{p_1} \hookrightarrow LM_{p_2, q}(1, w)$ which is multidimensional weighted Hardy's inequality. 

Since H\"{o}lder's inequality is strict, it is possible to obtain better results for the boundedness problem in local Morrey-type spaces, using the characterization of the embeddings between local Morrey-type spaces. With this motivation we started investigating the embedding problem between various function spaces. In \cite{MusUn}, characterizations of the embeddings between weighted Lebesgue spaces and weighted (complementary) local Morrey-type spaces are given with the help of solutions of the multidimensional direct and reverse weighted Hardy inequalities. In \cite{GogMusUn-LM}, embeddings between weighted complementary local Morrey-type spaces and weighted local Morrey-type spaces are investigated. The method was based on the duality approach which allowed us to reduce this problem to the characterizations of some iterated Hardy inequalities. It was not possible to solve the embedding \eqref{emb. LM1-LM2}, because it required the solutions of some iterated Hardy inequalities which were unknown at that time. These problems have been solved in \cite{KrePick}, recently. 

Although the same method works in our case, in this paper we will use a different approach to give the characterization of the embeddings \eqref{emb. LM1-LM2}. We will reduce this problem to the one-dimensional case and use the solutions of the embeddings between weighted Ces\`{a}ro function spaces from \cite{Unver} which can be seen as the one-dimensional case of our problem.

Let us briefly describe the structure of our paper. In the next section we give some information about the equivalency of some one-dimensional and $n$-dimensional inequalities. In the last section, we will provide some background information on the properties of weighted local Morrey-type spaces and characterize the embeddings between these spaces.

\section{Equivalent Inequalities}

The weight characterizations for one-dimensional Hardy inequality, that is,
\begin{equation*}
\bigg(\int_0^{\infty} \bigg(\int_0^x f(t) dt \bigg)^q w(x) dx \bigg)^{\frac{1}{q}} \leq C \bigg(\int_0^{\infty} f(t)^p v(t) dt \bigg)^{\frac{1}{p}}
\end{equation*}
have been studied for all $f\in \mathcal{M}^+(0, \infty)$, throughly, where $0 < q \leq \infty$ and $1 \leq p \leq \infty$. For a detailed history, see \cite{KufPerSam-Book}.

In \cite{EvGogOp}, the complete weight characterizations for so-called reverse Hardy inequality, that is,
\begin{equation*}
\bigg(\int_0^{\infty} \bigg(\int_0^x f(t) dt \bigg)^q w(x) dx \bigg)^{\frac{1}{q}} \geq C \bigg(\int_0^{\infty} f(t)^p v(t) dt \bigg)^{\frac{1}{p}}
\end{equation*} 
were given for all $ f\in \mathcal{M}^+(0, \infty)$ using the discretization method, where $0 < q \leq \infty$ and $0 < p \leq 1$.

In \cite{DrabHeinKuf}, using polar coordinates, authors extended the one-dimensional Hardy inequality to $n$-dimensional case and gave the characterizations of weights for which
\begin{equation*}
\bigg(\int_{\mathbb{R}^n} \bigg(\int_{B(0, |x|)} f(t) dt \bigg)^q w(x) dx \bigg)^{\frac{1}{q}} \leq C \bigg(\int_{\mathbb{R}^n} f(t)^p v(t) dt \bigg)^{\frac{1}{p}}
\end{equation*}
holds for all $f \in \mathcal{M}^+(\mathbb{R}^n)$, where $0 < q < \infty$ and $1 < p < \infty$. Note that in \cite{ChristGraf}, this problem has been considered for a special case. 

In \cite{GogMus-RevHardy}, authors dealt with the multidimensional analogue of reverse Hardy inequality, 
\begin{equation}
\bigg(\int_0^{\infty} \bigg(\int_{B(0, x)} f(t) dt \bigg)^q w(x) dx \bigg)^{\frac{1}{q}} \geq C \bigg(\int_{\mathbb{R}^n} f(t)^p v(t) dt \bigg)^{\frac{1}{p}}
\end{equation} 
for all $f \in \mathcal{M}^+(\mathbb{R}^n)$, where $0 < q \leq \infty$ and $0 < p \leq 1$. The approach they used here is the multidimensional analogues of the discretization technique given in \cite{EvGogOp}. 

Although these problems are in higher dimensions, it seems that with some simple modifications in one-dimensional theory, they can also be analysed using the similar steps. On that account it would be better to reduce the problems in the $n$-dimensional cases to the problems in one-dimesional cases. In \cite{Sinnamon}, Sinnamon handled the $n$-dimensional Hardy inequality by reducing it to one-dimensional Hardy inequality and extended the results on $n$-dimensional Hardy inequality to more general star shape domains. We will use ideas from \cite{Sinnamon} in this paper. 

Before proceeding to the theorems, let us recall the following integration in polar coordinates formula. By $S^{n-1}$, we denote the unit sphere $\{x \in \mathbb{R}^n: |x| = 1\}$. If $x \in \mathbb{R}^n \backslash \{0\} $, the polar coordinates of $x$ are 
\begin{equation*}
r= |x| \in (0, \infty), \quad x' = \frac{x}{|x|} \in S^{n-1}.
\end{equation*}
There is a unique Borel measure $\sigma = \sigma_{n-1}$ on $S^{n-1}$ such that if $f$ is Borel measurable on $\mathbb{R}^n$ and $f \geq 0$ or $f \in L_1(\mathbb{R}^n)$, then
\begin{equation*}
\int_{\mathbb{R}^n} f(x) dx = \int_0^{\infty} \int_{S^{n-1}} f(\tau x') \tau^{n-1} d\sigma(x') d\tau
\end{equation*}
(see, for instance, \cite{Folland}).

Let us now present our equivalency results. 

\begin{thm}\label{T:Morrey-Cesaro-p<1}
Let $0 < p < 1$ and $0 < q,\theta < \infty$. Assume that $u, w \in \mathcal{W}(0, \infty)$ and $v \in \mathcal{W}(\mathbb{R}^n) $.  Then the following two statements are equivalent.

\textup{(i)} There exists a constant $C$ such that	
\begin{equation}\label{Morrey emb.-p<1}
\bigg(\int_0^{\infty} \bigg(\int_{B(0,t)} f(s)^p v(s) ds\bigg)^{\frac{q}{p}} u(t) dt 
\bigg)^{\frac{1}{q}}  \leq C \bigg( \int_0^{\infty} \bigg(\int_{B(0,t)} f(s)  ds 
\bigg)^{\theta} w(t) dt \bigg)^{\frac{1}{\theta}}
\end{equation}
holds for every $f\in \mathcal{M}^+(\mathbb{R}^n)$.

\textup{(ii)} There exists a constant $C'$ such that
\begin{equation}\label{Cesaro emb.-p<1}
\bigg(\int_0^{\infty} \bigg(\int_0^t g(s)^p \tilde v(s) ds\bigg)^{\frac{q}{p}} u(t) dt 
\bigg)^{\frac{1}{q}}  \leq C' \bigg( \int_0^{\infty} \bigg(\int_0^t g(s)  ds 
\bigg)^{\theta} 
w(t) dt \bigg)^{\frac{1}{\theta}}
\end{equation}
holds for every $g\in \mathcal{M}^+(0,\infty)$, where
\begin{equation*}
\tilde v(t):= \bigg(\int_{S^{n-1}} v(ts')^{\frac{1}{1-p}} d\sigma(s')\bigg)^{1-p} 
t^{(n-1)(1-p)}.
\end{equation*}
Moreover the best constants of inequalities \eqref{Morrey emb.-p<1} and \eqref{Cesaro emb.-p<1} 
satisfies $C = C'$.
\end{thm}

\begin{proof}
Assume first that \eqref{Morrey emb.-p<1} holds for all $f \in \mathcal{M}^+(\mathbb{R}^n)$.
Set
$$
f(x) = g(|x|) \bigg(\int_{S^{n-1}} v(|x|\tau')^{\frac{1}{1-p}} d\sigma(\tau')\bigg)^{-1} 
v(x)^{\frac{1}{1-p}} |x|^{1-n}.
$$
Using spherical coordinates, we have that
\begin{align*}
LHS\eqref{Morrey emb.-p<1} & =  \bigg(\int_0^{\infty} \bigg( \int_0^t \int_{S^{n-1}} f(rs')^p v(rs') d\sigma(s')  
r^{n-1} dr \bigg)^{\frac{q}{p}} u(t) dt \bigg)^{\frac{1}{q}} \\
& = \bigg(\int_0^{\infty} \bigg(\int_0^t \int_{S^{n-1}} g(|rs'|)^p \bigg(\int_{S^{n-1}} v(|rs'|\tau')^{\frac{1}{1-p}} d\sigma(\tau')\bigg)^{-p}   \bigg.\bigg.\\
& \bigg. \bigg.  \hspace{3cm}\times v(rs')^{\frac{p}{1-p}+1} \,  |rs'|^{p(1-n)} d\sigma(s')  r^{n-1} dr \bigg)^{\frac{q}{p}} u(t) dt \bigg)^{\frac{1}{q}}\\
& = \bigg(\int_0^{\infty} \bigg(\int_0^t g(r)^p \bigg(\int_{S^{n-1}} v(r\tau')^{\frac{1}{1-p}} 
d\sigma(\tau')\bigg)^{1-p} r^{(n-1)(1-p)} dr \bigg)^{\frac{q}{p}} u(t) dt \bigg)^{\frac{1}{q}}\\
& = \bigg(\int_0^{\infty} \bigg(\int_0^t g(s)^p \tilde v(s) ds\bigg)^{\frac{q}{p}} u(t) dt 
\bigg)^{\frac{1}{q}}\\
& = LHS\eqref{Cesaro emb.-p<1}.
\end{align*}
Moreover, 
\begin{align*}
RHS\eqref{Morrey emb.-p<1} & = C \bigg(\int_0^{\infty} \bigg(\int_0^t \bigg[\int_{S^{n-1}} f(rs') d\sigma(s') \bigg] r^{n-1} dr \bigg)^{\theta} w(t) dt \bigg)^{\frac{1}{\theta}} \\
& = C \bigg(\int_0^{\infty} \bigg(\int_0^t g(r) dr \bigg)^{\theta} w(t) dt  
\bigg)^{\frac{1}{\theta}}.
\end{align*}
Therefore, \eqref{Cesaro emb.-p<1} holds with the best constant $C'$ such that  $C'\leq C$.
	
Conversely, assume now that \eqref{Cesaro emb.-p<1} holds for all $g \in \mathcal{M}^+(0,\infty)$.	
Applying H\"{o}lder's inequality with exponents $(\frac{1}{p}, \frac{1}{1-p})$, we get that
\begin{align*}
LHS\eqref{Morrey emb.-p<1} & \leq \bigg(\int_0^{\infty} \bigg(\int_0^t \bigg[\int_{S^{n-1}} 
f(rs')  d\sigma(s')\bigg]^p \bigg[\int_{S^{n-1}} v(rs')^{\frac{1}{1-p}} d\sigma(s')\bigg]^{1-p}  
r^{n-1} dr \bigg)^{\frac{q}{p}} u(t) dt \bigg)^{\frac{1}{q}}\\
& =  \bigg(\int_0^{\infty} \bigg(\int_0^t  \tilde{v}(r) \bigg[\int_{S^{n-1}} f(rs')^p  d\sigma(s') 
\bigg]^p  r^{p(n-1)} dr \bigg)^{\frac{q}{p}} u(t) dt \bigg)^{\frac{1}{q}}.
\end{align*}
Applying inequality \eqref{Cesaro emb.-p<1}, we obtain that
\begin{align*}
LHS\eqref{Morrey emb.-p<1} & \leq C' \bigg(\int_0^{\infty} \bigg(\int_0^t \bigg[\int_{S^{n-1}} 
f(rs') d\sigma(s')\bigg]  r^{(n-1)} dr \bigg)^{\theta} w(t) dt \bigg)^{\frac{1}{\theta}}\\
&= C' \bigg( \int_0^{\infty} \bigg(\int_{B(0,t)} f(s)  ds 
\bigg)^{\theta} w(t) dt \bigg)^{\frac{1}{\theta}}
\end{align*}
holds. Hence, \eqref{Morrey emb.-p<1} holds. Moreover the best constant of inequality \eqref{Morrey emb.-p<1} satisfy $C \leq C'$.  
\end{proof}

Let us now consider the case when $p=1$.

\begin{thm}\label{T:Morrey-Cesaro-p=1}
Let $0 < q, \theta < \infty$. Assume that $u, w \in \mathcal{W}(0, \infty)$ and $v \in \mathcal{W}(\mathbb{R}^n) $.  Then the following two statements are equivalent.

\textup{(i)} There exists a constant $C$ such that	
\begin{equation}\label{Morrey emb.p=1}
\bigg(\int_0^{\infty} \bigg(\int_{B(0,t)} f(s) v(s) ds\bigg)^{q} u(t) dt \bigg)^{\frac{1}{q}}  \leq 
C \bigg( \int_0^{\infty} \bigg(\int_{B(0,t)} f(s)  ds \bigg)^{\theta} w(t) dt 
\bigg)^{\frac{1}{\theta}}
\end{equation}
holds for every $f \in \mathcal{M}^+(\mathbb{R}^n)$.

\textup{(i)} There exists a constant $C'$ such that	
\begin{equation}\label{Cesaro emb.p=1}
\bigg(\int_0^{\infty} \bigg(\int_0^t g(s) \tilde v(s) ds\bigg)^{q} u(t) dt \bigg)^{\frac{1}{q}}  
\leq C' \bigg( \int_0^{\infty} \bigg(\int_0^t g(s)  ds \bigg)^{\theta} w(t) dt 
\bigg)^{\frac{1}{\theta}}
\end{equation}
holds for every $g \in \mathcal{M}^+(0, \infty)$, where
\begin{equation*}
\tilde v(t):= \esup_{s' \in S^{n-1}} v(ts').
\end{equation*}
Moreover the best constants of inequalities \eqref{Morrey emb.p=1} and \eqref{Cesaro emb.p=1} satisfy $C \approx C'$.
\end{thm}

\begin{proof}
Assume first that \eqref{Morrey emb.p=1} holds for all $f \in \mathcal{M}^+(\mathbb{R}^n)$.  Let 
$h\in \mathcal{M}^+(\mathbb{R}^n)$ be a function that saturates H\"{o}lder's inequality, that is, 
function satisfying
\begin{equation}\label{function h}
\supp h \subset S^{n-1}, \quad \int_{S^{n-1}} h(rs') d\sigma(s') = 1
\end{equation} 
and 
\begin{equation}\label{function h rev Hol.}
\int_{S^{n-1}} h(rs') v(rs') d\sigma(s') \gtrsim \tilde v.
\end{equation}
Then we define  
\begin{equation}\label{test function}
f(x) = h(x) g(|x|) |x|^{1-n}.
\end{equation}
Using spherical coordinates, we have that
\begin{align*}
LHS\eqref{Morrey emb.p=1} &=  \bigg(\int_0^{\infty} \bigg(\int_0^t \int_{S^{n-1}} f(rs') 
v(rs') d\sigma(s') r^{n-1} dr \bigg)^q u(t) dt \bigg)^{\frac{1}{q}} \\
&=\bigg(\int_0^{\infty} \bigg(\int_0^t \int_{S^{n-1}}  h(rs') g(r) r^{1-n}v(rs') d\sigma(s')  r^{n-1} dr \bigg)^q u(t) dt \bigg)^{\frac{1}{q}} \\
&=\bigg(\int_0^{\infty} \bigg(\int_0^t g(r) \bigg[\int_{S^{n-1}}  h(rs') v(rs') d\sigma(s') 
\bigg] dr \bigg)^q u(t) dt \bigg)^{\frac{1}{q}} \\
\end{align*}
Applying \eqref{function h rev Hol.}, we obtain that
\begin{align*}
LHS\eqref{Morrey emb.p=1}  \gtrsim \bigg(\int_0^{\infty} \bigg(\int_0^t g(r) \tilde{v}(r) dr \bigg)^q u(t) dt \bigg)^{\frac{1}{q}} = LHS\eqref{Cesaro emb.p=1}.	
\end{align*}

On the other hand, using spherical coordinates and \eqref{function h}, 
\begin{align*}
RHS\eqref{Morrey emb.p=1} &= C \bigg(\int_0^{\infty} \bigg(\int_0^t \int_{S^{n-1}} h(rs') g(r) r^{1-n} d\sigma(s') 
r^{n-1} dr \bigg)^{\theta} w(t) dt \bigg)^{\frac{1}{\theta}} \\
&= C \bigg(\int_0^{\infty} \bigg(\int_0^t g(r) \bigg[\int_{S^{n-1}} h(rs') d\sigma(s') \bigg] dr 
\bigg)^{\theta} w(t) dt \bigg)^{\frac{1}{\theta}} \\
& = C \bigg(\int_0^{\infty} \bigg(\int_0^t g(r) dr \bigg)^{\theta} w(t) dt 
\bigg)^{\frac{1}{\theta}}
\end{align*}
holds. Therefore, the following chain of relations is true. 
\begin{equation*}
LHS\eqref{Cesaro emb.p=1} \lesssim LHS\eqref{Morrey emb.p=1} \leq C RHS\eqref{Morrey emb.p=1} =  C 
\bigg(\int_0^{\infty} \bigg(\int_0^t g(r) dr \bigg)^{\theta} w(t) dt 
\bigg)^{\frac{1}{\theta}}.
\end{equation*}
As a result, \eqref{Cesaro emb.p=1} holds and the best constant of inequality \eqref{Cesaro emb.p=1} satisfies $C' \lesssim C$.
	
Assume now that \eqref{Cesaro emb.p=1} holds for all $g \in \mathcal{M}^+(0,\infty)$. Using spherical coordinates 
again, we have that
\begin{align*}
LHS\eqref{Morrey emb.p=1} 
&= \bigg(\int_0^{\infty} \bigg(\int_0^t \int_{S^{n-1}} f(rs') v(rs') d\sigma(s') r^{n-1} dr \bigg)^q u(t) dt \bigg)^{\frac{1}{q}}\\ 
& \leq \bigg(\int_0^{\infty} \bigg(\int_0^t \esup_{s' \in S^{n-1}} v(rs') \bigg[\int_{S^{n-1}} 		
f(rs')	d \sigma(s') \bigg]	r^{n-1} dr \bigg)^q u(t) dt \bigg)^{\frac{1}{q}}\\
& =  \bigg(\int_0^{\infty} \bigg(\int_0^t  \tilde{v}(r) \bigg[\int_{S^{n-1}} f(rs') 
d\sigma(s')\bigg]  r^{n-1} dr \bigg)^q u(t) dt \bigg)^{\frac{1}{q}}.
\end{align*}
Applying inequality \eqref{Cesaro emb.p=1}, 
\begin{align*}
LHS\eqref{Morrey emb.p=1}  &\leq C' \bigg(\int_0^{\infty} \bigg(\int_0^t  \bigg[\int_{S^{n-1}} 
f(rs') d\sigma(s')\bigg]  r^{(n-1)} dr \bigg)^{\theta} w(t) dt \bigg)^{\frac{1}{\theta}}\\
&= C' \bigg( \int_0^{\infty} \bigg(\int_{B(0,t)} f(s)  ds \bigg)^{\theta} w(t) dt 
\bigg)^{\frac{1}{\theta}}
\end{align*}
holds. Moreover, the best constant $C$ of inequality \eqref{Morrey emb.p=1} satisfies $C \leq C'$.
\end{proof}

\begin{rem}\label{Fubini result}
We should note that when $\theta=1$ or $p= q$,  using Fubini's Theorem, inequality \eqref{Morrey emb.-p<1} coincides with some $n$-dimensional Hardy inequality and $n$-dimensional reverse Hardy inequality, respectively.  
\end{rem}

\section{Embeddings}
The aim of this section is to characterize the embeddings between weighted local Morrey-type spaces, 
that is, \eqref{emb. LM1-LM2}.

We will begin with the formulation of some properties of these spaces. 
\begin{rem}
In \cite{BurGul2004}, it was proved that for $0 < p, q \leq \infty$ and $w \in \mathcal{M}^+(0,\infty)$, if $\|w\|_{q,(t,\infty)}= \infty$ for all $t > 0$, then $LM_{p,q}(1, w)$ consists only of function equivalent to $0$ on $\mathbb{R}^n$. The same conclusion is true for $LM_{p,q}(v,w)$ for any $v \in \mathcal{W}(\mathbb{R}^n)$. Therefore we will always assume that $ \|w\|_{q, (t, \infty)} < \infty$. 
\end{rem}

\begin{rem}
We can formulate Remark~\ref{Fubini result} also in the following way: Let $0 < p \leq \infty$ and $v \in \mathcal{W}(\mathbb{R}^n)$. Then $LM_{p, p}(v, w) = L_p(u)$, where $u(x) := v(x) \, \|w\|_{p, (|x|, \infty)}$, $x \in \mathbb{R}^n$. With this result  it is clear that when $p_1 = q_1$ or $p_2 = q_2 $, \eqref{emb. LM1-LM2} coincides with the embeddings between weighted Lebesgue spaces and weighted local Morrey-type spaces, which were given in \cite{MusUn}. 
\end{rem}

Unfortunately, we will solve this embedding problem under the restriction $p_2 < q_2$, we will deal with the remaining cases in a future paper.

In order to shorten the expressions we will use the following notation:
\begin{equation*}
LM_i := LM_{p_i, q_i}(v_i,w_i), \quad i=1,2. 
\end{equation*}

\begin{thm}\label{maintheorem1}
Let $0 < q_1 \leq p_2 < \min\{p_1, q_2\}$. Assume that $v_1, v_2 \in \mathcal{W}(\mathbb{R}^n)$ and $w_1, w_2 \in \mathcal{W}(0,\infty)$ such that $\int_t^{\infty} w_i^{q_i} < \infty$, $i=1,2$ for all $t \in (0, \infty)$. 
	
{\rm (i)} If $p_1 \leq q_2 < \infty$, then $LM_1 \hookrightarrow LM_2$ for all $f \in \mathcal{M}^+(\mathbb{R}^n)$ if and only if $I_1 < \infty$, where
\begin{align*}
I_1 := \sup_{x \in (0,\infty)} \bigg(\int_x^{\infty} w_1^{q_1} \bigg)^{-\frac{1}{q_1}} \sup_{t	\in (x,\infty)} \bigg(\int_{B(x,t)} v_1^{-\frac{p_1p_2}{p_1-p_2}} v_2^{\frac{p_1p_2}{p_1-p_2}} \bigg)^{\frac{p_1-p_2}{p_1p_2}} \bigg( \int_t^{\infty} w_2^{q_2} \bigg)^{\frac{1}{q_2}}.
\end{align*}
Moreover, $\|\Id\|_{LM_1 \rightarrow LM_2} \approx I_1$.

{\rm (ii)} If $q_2 < p_1< \infty$, then  $LM_1 \hookrightarrow LM_2$ for all $f \in \mathcal{M}^+(\mathbb{R}^n)$ if and only if  $I_2 < \infty$, where
\begin{align*}
I_2& :=  \sup_{x \in (0,\infty)} \bigg(\int_x^{\infty} w_1^{q_1} \bigg)^{-\frac{1}{q_1}} \bigg(\int_x^{\infty} \bigg(\int_{B(x,t)} v_1^{-\frac{p_1p_2}{p_1-p_2}} v_2^{\frac{p_1p_2}{p_1-p_2}}
\bigg)^{\frac{q_2(p_1-p_2)}{p_2(p_1-q_2)}} \bigg.\\
&\hspace{3cm} \bigg. \times \bigg( \int_t^{\infty} w_2^{q_2} \bigg)^{\frac{q_2}{p_1-q_2}} w_2^{q_2}(t) dt \bigg)^{\frac{p_1-q_2}{p_1q_2}}.
\end{align*}
Moreover, $\|\Id\|_{LM_1 \rightarrow LM_2} \approx I_2$.
\end{thm}

\begin{proof}
Since 
\begin{align*}
\|\Id\|_{LM_1\rightarrow LM_2} &= \sup_{f \in \mathcal{M}^+(\mathbb{R}^n)} \ddfrac{\|f\|_{LM_2}}{\|f\|_{LM_1}}\\
&= \sup_{f \in \mathcal{M}^+(\mathbb{R}^n)} \ddfrac{\bigg(\int_0^{\infty} \bigg(\int_{B(0,t)} f(s)^{p_2} v_2(s)^{p_2} ds\bigg)^{\frac{q_2}{p_2}} w_2(t)^{q_2} dt 
\bigg)^{\frac{1}{q_2}}}{\bigg(\int_0^{\infty} \bigg(\int_{B(0,t)} f(s)^{p_1} v_1(s)^{p_1} ds\bigg)^{\frac{q_1}{p_1}} w_1(t)^{q_1} dt 
\bigg)^{\frac{1}{q_1}}}\\
&= \sup_{g \in \mathcal{M}^+(\mathbb{R}^n)} \ddfrac{\bigg(\int_0^{\infty} \bigg(\int_{B(0,t)} g(s)^{\frac{p_2}{p_1}}v_1(s)^{-p_2} v_2(s)^{p_2} ds \bigg)^{\frac{q_2}{p_2}} w_2(t)^{q_2} dt 
\bigg)^{\frac{1}{q_2}}}{\bigg(\int_0^{\infty} \bigg( \int_{B(0,t)} g(s) ds \bigg)^{\frac{q_1}{p_1}} w_1(t)^{q_1} dt \bigg)^{\frac{1}{q_1}}}.
\end{align*}

Taking parameters 
\begin{equation*}
p:= \frac{p_2}{p_1}, \quad q:= \frac{q_2}{p_1}, \quad q:= \frac{q_1}{p_1},
\end{equation*}
and weights,
\begin{equation*}
u:= w_2^{q_2}, \quad v:= v_1^{-p_2} v_2^{p_2}, \quad w:= w_1^{q_1}, 
\end{equation*}
we obtain that
\begin{equation*}
\|\Id\|_{LM_1\rightarrow LM_2}^{p_1} = \sup_{g \in \mathcal{M}^+(\mathbb{R}^n)} \ddfrac{\bigg(\int_0^{\infty} \bigg(\int_{B(0,t)} g(s)^p v(s) ds \bigg)^{\frac{q}{p}} u(t) dt 	\bigg)^{\frac{1}{q}}} {\bigg( \int_0^{\infty} \bigg( \int_{B(0,t)} g(s) ds \bigg)^{q} w(t) dt \bigg)^{\frac{1}{q}}}.
\end{equation*}

By Theorem~\ref{T:Morrey-Cesaro-p<1}, we have that
\begin{equation*}
\|\Id\|_{LM_1\rightarrow LM_2}^{p_1} = \sup_{h \in \mathcal{M}^+(0,\infty)} \ddfrac{\bigg(\int_0^{\infty} \bigg(\int_0^t h(s)^p \tilde v(s) ds \bigg)^{\frac{q}{p}} u(t) dt 	\bigg)^{\frac{1}{q}}} {\bigg( \int_0^{\infty} \bigg( \int_0^t h(s) ds \bigg)^{\theta} w(t) dt \bigg)^{\frac{1}{\theta}}},
\end{equation*}
where
\begin{equation*}
\tilde v(\tau):= \bigg(\int_{S^{n-1}} v(\tau s')^{\frac{1}{1-p}} d\sigma(s')\bigg)^{1-p} \tau^{(n-1)(1-p)}.
\end{equation*}

{\rm (i)} If $p_1 \leq q_2 < \infty$, then $1 \leq q < \infty$. Therefore, applying \cite[Theorem 1.1, (i)]{Unver}, we have that
\begin{align*}
\|\Id\|_{LM_1\rightarrow LM_2}^{p_1} & \approx  \sup_{x \in (0,\infty)} \bigg(\int_x^{\infty} w \bigg)^{-\frac{1}{\theta}} 
\sup_{t \in (x,\infty)} \bigg(\int_x^t \tilde{v}^{\frac{1}{1-p}} \bigg)^{\frac{1-p}{p}} \bigg( \int_t^{\infty} u\bigg)^{\frac{1}{q}} \\
&\hspace{-0.2cm}= \sup_{x \in (0,\infty)} \bigg(\int_x^{\infty} w \bigg)^{-\frac{1}{\theta}} \sup_{t \in (x,\infty)} \bigg(\int_x^t \tau^{n-1} \int_{S^{n-1}} v(\tau s')^{\frac{1}{1-p}} d\sigma(s') d\tau \bigg)^{\frac{1-p}{p}} \bigg(\int_t^{\infty} u \bigg)^{\frac{1}{q}}\\
&\hspace{-0.2cm} = \sup_{x \in (0,\infty)} \bigg(\int_x^{\infty} w \bigg)^{-\frac{1}{\theta}} \sup_{t \in (x,\infty)} \bigg(\int_{B(x, t)}  v(y)^{\frac{1}{1-p}} dy \bigg)^{\frac{1-p}{p}} \bigg(\int_t^{\infty} u \bigg)^{\frac{1}{q}}.
\end{align*}
Now, by substituting parameters $p,q, \theta$ and weights $u, v, w$ into the last expression, we arrive at
\begin{align*}
\|\Id\|_{LM_1\rightarrow LM_2}  \approx \sup_{x \in (0,\infty)} \bigg(\int_x^{\infty} w_1^{q_1} \bigg)^{-\frac{1}{q_1}} \sup_{t	\in (x,\infty)} \bigg(\int_{B(x,t)} v_1^{-\frac{p_1p_2}{p_1-p_2}} v_2^{\frac{p_1p_2}{p_1-p_2}} \bigg)^{\frac{p_1-p_2}{p_1p_2}} \bigg( \int_t^{\infty} w_2^{q_2} \bigg)^{\frac{1}{q_2}}.  
\end{align*}

{\rm (ii)}  If $q_2 < p_1$, then $ q < 1$. Therefore, applying \cite[Theorem 1.1, (ii)]{Unver}, we have that
\begin{align*}
\|\Id\|_{LM_1\rightarrow LM_2}^{p_1} &\approx  \sup_{x \in (0,\infty)} \bigg(\int_x^{\infty} w \bigg)^{-\frac{1}{\theta}} 
\bigg(\int_x^{\infty} \bigg(\int_x^t \tilde{v}^{\frac{1}{1-p}} \bigg)^{\frac{q(1-p)}{p(1-q)}} 
\bigg( \int_t^{\infty} u \bigg)^{\frac{q}{1-q}} u(t) dt \bigg)^{\frac{1-q}{q}}\\
& =\sup_{x \in (0,\infty)} \bigg(\int_x^{\infty} w \bigg)^{-\frac{1}{\theta}} \bigg(\int_x^{\infty} 
\bigg(\int_x^t \tau^{n-1} \int_{S^{n-1}} v(\tau s')^{\frac{1}{1-p}} d\sigma(s') 
d\tau\bigg)^{\frac{q(1-p)}{p(1-q)}}  \bigg.\\
&\hspace{3cm} \bigg. \times \bigg( \int_t^{\infty} u \bigg)^{\frac{q}{1-q}} u(t) dt 
\bigg)^{\frac{1-q}{q}}\\ 
& =\sup_{x \in (0,\infty)} \bigg(\int_x^{\infty} w \bigg)^{-\frac{1}{\theta}} \bigg(\int_x^{\infty} 
\bigg(\int_{B(x,t)} v^{\frac{1}{1-p}} \bigg)^{\frac{q(1-p)}{p(1-q)}} \bigg(\int_t^{\infty} u 
\bigg)^{\frac{q}{1-q}} u(t) dt \bigg)^{\frac{1-q}{q}}.
\end{align*}
Similarly, substituting parameters $p,q, \theta$ and weights $u, v, w$ gives
\begin{align*}
\|\Id\|_{LM_1 \rightarrow LM_2} & \approx  \sup_{x \in (0,\infty)} \bigg(\int_x^{\infty} w_1^{q_1} \bigg)^{-\frac{1}{q_1}} \bigg(\int_x^{\infty} \bigg(\int_{B(x,t)} v_1^{-\frac{p_1p_2}{p_1-p_2}} v_2^{\frac{p_1p_2}{p_1-p_2}}
\bigg)^{\frac{q_2(p_1-p_2)}{p_2(p_1-q_2)}} \bigg.\\
&\hspace{2cm} \bigg. \times \bigg( \int_t^{\infty} w_2^{q_2} \bigg)^{\frac{q_2}{p_1-q_2}} w_2^{q_2}(t) dt \bigg)^{\frac{p_1-q_2}{p_1q_2}}.
\end{align*}
\end{proof}

\begin{thm}\label{maintheorem3}
Let $0 < p_2 < \min\{p_1, q_1, q_2\}$. Assume that $v_1, v_2 \in \mathcal{W}(\mathbb{R}^n)$ and $w_1, w_2 \in \mathcal{W}(0,\infty)$ such that $\int_t^{\infty} w_i^{q_i} < \infty$, $i=1,2$ for all $t \in (0, \infty)$ .  Suppose that
\begin{equation*}
0 < \int_0^t \bigg(\int_{B(s,t)} v_1^{-\frac{p_1p_2}{p_1-p_2}} v_2^{\frac{p_1p_2}{p_1-p_2}} \bigg)^{\frac{q_1(p_1-p_2)}{p_1(q_1 - p_2)}} \bigg(\int_s^{\infty} w_1^{q_1} \bigg)^{- \frac{q_1}{q_1-p_2}} w_1(s)^{q_1}\, ds  < \infty
\end{equation*}
holds for all $t \in (0,\infty)$.
	
{\rm (i)} If $\max\{p_1, q_1\} \leq q_2 < \infty$, then $LM_1 \hookrightarrow LM_2$ for all $f \in \mathcal{M}^+(\mathbb{R}^n)$ if and only if $I_3 < \infty$ and $I_4 < \infty$, where
\begin{align}\label{A_4}
I_3 :=  \bigg(\int_0^{\infty} w_1^{q_1} \bigg)^{-\frac{1}{q_1}} \sup_{t \in (0,\infty)} \bigg(\int_{B(0,t)} 
v_1^{-\frac{p_1p_2}{p_1-p_2}} v_2^{\frac{p_1p_2}{p_1-p_2}} \bigg)^{\frac{p_1-p_2}{p_1 p_2}} \bigg( \int_t^{\infty} w_2^{q_2} \bigg)^{\frac{1}{q_2}}
\end{align}
and
\begin{align}\label{A_{5}}
I_4 &:= \sup_{t \in (0,\infty)} \bigg( \int_t^{\infty} w_2^{q_2} \bigg)^{\frac{1}{q_2}}  \bigg( \int_0^t \bigg(\int_{B(s,t)} 
v_1^{-\frac{p_1p_2}{p_1-p_2}} v_2^{\frac{p_1p_2}{p_1-p_2}} \bigg)^{\frac{q_1(p_1-p_2)}{p_1(q_1 - p_2)}} \bigg.\notag\\
&\hspace{2cm} \times \bigg. \bigg(\int_s^{\infty} w_1^{q_1} \bigg)^{-\frac{q_1}{q_1 - p_2}} w_1(s)^{q_1} ds \bigg)^{\frac{p_1-q_2}{q_1 q_2}}.
\end{align}
Moreover, $\|\Id\|_{LM_1 \rightarrow LM_2} \approx I_3 + I_4$.

{\rm (ii)} If $p_1 \leq q_2 < q_1 < \infty$,  then $LM_1 \hookrightarrow LM_2$ for all $f \in \mathcal{M}^+(\mathbb{R}^n)$ if and only if $I_3 < \infty$, $I_5 < \infty$ and $I_6 < \infty$, where $I_3$ is defined in \eqref{A_4},  
\begin{align*}
I_5 &:=  \left( \int_0^{\infty} \bigg( \int_0^t \bigg( \int_s^{\infty} w_1^{q_1} \bigg)^{-\frac{q_1}{q_1 - p_2}} w_1(s)^{q_1} ds \bigg)^{\frac{q_1(q_2-p_2)}{p_2(q_1 - q_2)}}  \bigg( \int_t^{\infty} w_1^{q_1} \bigg)^{-\frac{q_1}{q_1 - p_2}} w_1(t)^{q_1} \right. 
\notag\\
& \hspace{2cm} \times \left. \sup_{z \in (t,\infty)} \bigg(\int_{B(t,z)} v_1^{-\frac{p_1p_2}{p_1-p_2}} v_2^{\frac{p_1p_2}{p_1-p_2}} \bigg)^{\frac{q_1 q_2(p_1-p_2)}{p_1 p_2(q_1 - q_2)}} \bigg( \int_z^{\infty} w_2^{q_2} \bigg)^{\frac{q_1}{q_1 - q_2}} dt \right)^{\frac{q_1-q_2}{q_1 q_2}}
\end{align*}
and
\begin{align}\label{A_{7}}
I_6 &:=  \left( \int_0^{\infty} \bigg( \int_0^t \bigg( \int_s^{\infty} w_1^{q_1} \bigg)^{-\frac{q_1}{q_1 - p_2}} w_1(s)^{q_1} 
\bigg(\int_{B(s,t)} v_1^{-\frac{p_1p_2}{p_1-p_2}} v_2^{\frac{p_1p_2}{p_1-p_2}} \bigg)^{\frac{q_1(q_2-p_2)}{p_1(q_1 - p_2)}} ds \bigg)^{\frac{q_1(q_2-p_2)}{p_2(q_1 - 
q_2)}} \right. \notag\\
& \hspace{0.3cm} \times \left. \sup_{z \in (t,\infty)} \bigg(\int\limits_{B(t,z)} v_1^{-\frac{p_1p_2}{p_1-p_2}} v_2^{\frac{p_1p_2}{p_1-p_2}}
\bigg)^{\frac{q_1(p_1-p_2)}{p_1(q_1 - p_2)}} \bigg( \int\limits_z^{\infty} w_2^{q_2} \bigg)^{\frac{q_1}{q_1 - q_2}}   \bigg( \int_t^{\infty} w_1^{q_1} \bigg)^{-\frac{q_1}{q_1 - p_2}} w_1(t)^{q_1}   dt \right)^{\frac{q_1-q_2}{q_1 q_2}}.
\end{align}
Moreover, $\|\Id\|_{LM_1 \rightarrow LM_2} \approx  I_3 + I_5 + I_6$.

{\rm (iii)} If $q_1 \leq q_2 < p_1< \infty$, then $LM_1 \hookrightarrow LM_2$ for all $f \in \mathcal{M}^+(\mathbb{R}^n)$ if and only if $I_4 < \infty$, $I_7 < \infty$ and $I_8 < \infty$, where $I_4$ is defined in \eqref{A_{5}}, 
\begin{align}\label{A_{8}}
I_7 :=  \bigg( \int\limits_0^{\infty} w_1^{q_1} \bigg)^{-\frac{1}{q_1}} \bigg( \int_0^{\infty} \bigg(\int_{B(0,t)} v_1^{-\frac{p_1p_2}{p_1-p_2}} v_2^{\frac{p_1p_2}{p_1-p_2}} \bigg)^{\frac{q_2(p_1-p_2)}{p_2(p_1-q_2)}} \bigg( \int_t^{\infty} w_2^{q_2} \bigg)^{\frac{q_2}{p_1- q_2}} w_2(t)^{q_2} dt 
\bigg)^{\frac{p_1-q_2}{p_1 q_2}}
\end{align}
and
\begin{align*}
I_8 &:=  \sup_{t \in (0,\infty)} \bigg( \int_0^t \bigg( \int_s^{\infty} w_1^{q_1} 
\bigg)^{-\frac{q_1}{q_1 - p_2}} w_1(s)^{q_1} ds \bigg)^{\frac{q_1-p_2}{q_1 p_2}} \\
&\hspace{2cm} \times \bigg( \int_t^{\infty} \bigg(\int_{B(t,s)} v_1^{-\frac{p_1p_2}{p_1-p_2}} v_2^{\frac{p_1p_2}{p_1-p_2}} \bigg)^{\frac{q_2(p_1-p_2)}{p_2(p_1-q_2)}} \bigg( \int_s^{\infty} w_2^{q_2} \bigg)^{\frac{q_2}{p_1-q_2}} w_2(s)^{q_2} ds 
\bigg)^{\frac{p_1-q_2}{p_1q_2}}.
\end{align*} Moreover, $\|\Id\|_{LM_1 \rightarrow LM_2} \approx  I_4 + I_7 + I_8$.

{\rm (iv)} If $p_1 < \infty$, $q_1 < \infty$ and $ q_2 < \min\{p_1,q_1\}$, then $LM_1 \hookrightarrow LM_2$ for all $f \in \mathcal{M}^+(\mathbb{R}^n)$ if and only if  $I_6 < \infty$, $I_{7} < \infty$ and $I_9 < \infty$, where $I_6$ and $I_7$ are defined in \eqref{A_{7}} and \eqref{A_{8}}, respectively, and
\begin{align*}
I_9 &:=  \left( \int_0^{\infty} \bigg( \int_0^t \bigg( \int_s^{\infty} w_1^{q_1} \bigg)^{-\frac{q_1}{q_1 - p_2}} w_1(s)^{q_1} ds \bigg)^{\frac{q_1(q_2-p_2)}{p_2(q_1 - q_2)}}  
\bigg( \int_t^{\infty} w_1^{q_1} \bigg)^{-\frac{q_1}{q_1 - p_2}} w_1(t)^{q_1} \right. \notag\\
& \hspace{1cm} \times \left. \bigg( \int_t^{\infty} \bigg(\int_{B(t,s)}v_1^{-\frac{p_1p_2}{p_1-p_2}} v_2^{\frac{p_1p_2}{p_1-p_2}} \bigg)^{\frac{q_2(p_1-p_2)}{p_2(p_1-q_2)}} \bigg( \int_s^{\infty} w_2^{q_2} \bigg)^{\frac{q_1}{p_1 - q_2}} w_2(s)^{q_2} ds \bigg)^{\frac{q_1(p_1-q_2)}{q_1-q_2}} dt\right)^{\frac{q_1-q_2}{q_1 q_2}}.
\end{align*}
Moreover, $\|\Id\|_{LM_1 \rightarrow LM_2} \approx I_6 + I_7 +  I_9$.
\end{thm}

\begin{proof}
As in the proof of Theorem~\ref{maintheorem1}, using Theorem~\ref{T:Morrey-Cesaro-p<1} and applying  ~\cite[Theorem~1.3]{Unver} the result follows. 
\end{proof}

\begin{thm}\label{maintheorem2}
Let $0 < q_1 \leq p_1 = p_2 <q_2 < \infty$.  Assume that $v_1, v_2 \in \mathcal{W}(\mathbb{R}^n)$ and $w_1, w_2 \in \mathcal{W}(0,\infty)$ such that $\int_t^{\infty} w_i^{q_i} < \infty$, $i=1,2$ for all $t \in (0, \infty)$. Then, $LM_1 \hookrightarrow LM_2$ for all $f \in \mathcal{M}^+(\mathbb{R}^n)$ if and only if $I_{10} < \infty$, where
\begin{align*}
I_{10} := \sup_{x \in (0,\infty)} \bigg( \int_x^{\infty} w_2(s)^{q_2} \bigg)^{\frac{1}{q_2}} \esup_{y \in B(0, x)} v_1(y)^{-p_1} v_2(y)^{p_1} \bigg(\int_{|y|}^{\infty} w_1(s)^{q_1} ds \bigg)^{-\frac{1}{q_1}}.
\end{align*}
Moreover, $\|\Id\|_{LM_1 \rightarrow LM_2} \approx I_{10}$.
\end{thm}

\begin{proof}
As in the proof of Theorem~\ref{maintheorem1},	since 
\begin{align*}
\|\Id\|_{LM_1\rightarrow LM_2} &= \sup_{g \in \mathcal{M}^+(\mathbb{R}^n)} \ddfrac{\bigg(\int_0^{\infty} \bigg(\int_{B(0,t)} g(s) v_1(s)^{-p_1} v_2(s)^{p_1} ds \bigg)^{\frac{q_2}{p_1}} w_2(t)^{q_2} dt 	\bigg)^{\frac{1}{q_2}}}{\bigg(\int_0^{\infty} \bigg( \int_{B(0,t)} g(s) ds \bigg)^{\frac{q_1}{p_1}} w_1(t)^{q_1} dt \bigg)^{\frac{1}{q_1}}}
\end{align*}
taking parameters 
\begin{equation*}
p:= 1, \quad q:= \frac{q_2}{p_1}, \quad \theta:= \frac{q_1}{p_1},
\end{equation*}
and weights,
\begin{equation*}
u:= w_2^{q_2}, \quad v:= v_1^{-p_1} v_2^{p_1}, \quad w:= w_1^{q_1}, 
\end{equation*}
we obtain that
\begin{equation*}
\|\Id\|_{LM_1\rightarrow LM_2}^{p_1} = \sup_{g \in \mathcal{M}^+(\mathbb{R}^n)} \ddfrac{\bigg(\int_0^{\infty} \bigg(\int_{B(0,t)} g(s) v(s) ds \bigg)^q u(t) dt 	\bigg)^{\frac{1}{q}}} {\bigg( \int_0^{\infty} \bigg( \int_{B(0,t)} g(s) ds \bigg)^{\theta} w(t) dt \bigg)^{\frac{1}{\theta}}}.
\end{equation*}
	
By Theorem~\ref{T:Morrey-Cesaro-p=1}, we have that
\begin{equation}\label{main inequality equiv. p=1}
\|\Id\|_{LM_1\rightarrow LM_2}^{p_1} = \sup_{h \in \mathcal{M}^+(0, \infty)} \ddfrac{\bigg(\int_0^{\infty} \bigg(\int_0^t h(s) \tilde v(s) ds \bigg)^q u(t) dt 	\bigg)^{\frac{1}{q}}} {\bigg( \int_0^{\infty} \bigg( \int_0^t h(s) ds \bigg)^{\theta} w(t) dt \bigg)^{\frac{1}{\theta}}},
\end{equation}
where
\begin{equation*}
\tilde v(\tau):= \esup_{s' \in S^{n-1}} v(\tau s').
\end{equation*}

If $0 < q_1 \leq p_1 = p_2 <q_2 < \infty$, then $p=1$, $0 < \theta \leq  1 < q < \infty$. Therefore, applying \cite[Theorem 1.2]{Unver}, we have that
\begin{align*}
\|\Id\|_{LM_1\rightarrow LM_2}^{p_1} & \approx  \sup_{x \in (0,\infty)} \bigg(\int_x^{\infty} u \bigg)^{\frac{1}{q}} \esup_{\tau \in (0,x)} \tilde{v}(\tau) \bigg( \int_{\tau}^{\infty} w \bigg)^{-\frac{1}{\theta}} \\
& =  \sup_{x \in (0,\infty)} \bigg(\int_x^{\infty} u \bigg)^{\frac{1}{q}} \esup_{\tau \in (0,x)} \bigg(\esup_{s' \in S^{n-1}} v(\tau s') \bigg) \bigg( \int_{\tau}^{\infty} w \bigg)^{-\frac{1}{\theta}}.
\end{align*}
Observe that 
\begin{equation}\label{1}
\esup_{s' \in S^{n-1}} v(\tau s') = \esup_{|x|=\tau} v(x).
\end{equation}
Then, 
\begin{align*}
\|\Id\|_{LM_1\rightarrow LM_2}^{p_1} & \approx  \sup_{x \in (0,\infty)} \bigg(\int_x^{\infty} u \bigg)^{\frac{1}{q}} \esup_{\tau \in (0,x)} \bigg(\esup_{|x| = \tau} v(x) \bigg) \bigg( \int_{\tau}^{\infty} w \bigg)^{-\frac{1}{\theta}} \\
& =  \sup_{x \in (0,\infty)} \bigg(\int_x^{\infty} u \bigg)^{\frac{1}{q}} \esup_{\tau \in (0,x)} \esup_{|x| = \tau} v(x) \bigg( \int_{|x|}^{\infty} w \bigg)^{-\frac{1}{\theta}}\\
& = \sup_{x \in (0,\infty)} \bigg(\int_x^{\infty} u \bigg)^{\frac{1}{q}} \esup_{y \in B(0,x)}  v(y) \bigg( \int_{|y|}^{\infty} w \bigg)^{-\frac{1}{\theta}}.
\end{align*}
Now, substituting the powers $p, q, \theta$, and weights $u, v, w$, we obtain that
\begin{align*}
\|\Id\|_{LM_1 \rightarrow LM_2} \approx  \sup_{x \in (0,\infty)} \bigg( \int_x^{\infty} w_2(s)^{q_2} \bigg)^{\frac{1}{q_2}} \esup_{y \in B(0, x)} v_1(y)^{-p_1} v_2(y)^{p_1} \bigg(\int_{|y|}^{\infty} w_1(s)^{q_1} ds \bigg)^{-\frac{1}{q_1}}.   
\end{align*} 
\end{proof}

\begin{thm}\label{maintheorem4}
Let $0 < p_1=p_2 <  \min\{q_1, q_2\}$, $q_1, q_2 < \infty$. Assume that $v_1, v_2 \in \mathcal{W}(\mathbb{R}^n)$ such that $v_1^{-1}v_2$ is continuous and $w_1, w_2 \in \mathcal{W}(0,\infty)$ such that $\int_t^{\infty} w_i^{q_i} < \infty$, $i=1,2$ for all $t \in (0, \infty)$. Suppose that 
\begin{itemize}
\item  $0 < \int_0^t \esup_{|x|=\tau} v(x)^{\frac{q_1}{q_1 - p_1}} d\tau < \infty$, 
\item $ 0 < \int_0^t \bigg(\int_x^{\infty} w_1(s)^{q_1} ds \bigg)^{-\frac{q_1}{q_1-p_1}} w_1(x)^{q_1} dx < \infty$, 
\item $0 < \int_0^t w_2(s)^{-\frac{q_2p_1}{q_2 - p_1}} ds < \infty$
\end{itemize}
hold for all $t \in (0,\infty)$. 
		
{\rm (i)} If $q_1 \leq q_2$, then $LM_1 \hookrightarrow LM_2$ for all $f \in \mathcal{M}^+(\mathbb{R}^n)$ if and only if $I_{11} < \infty$ and $I_{12} < \infty$, where
\begin{align}\label{A_12}
I_{11} := \bigg(\int_0^{\infty} w_1^{q_1} \bigg)^{-\frac{1}{q_1}} \sup_{t \in (0,\infty)} \bigg( \int_t^{\infty} w_2^{q_2} 
\bigg)^{\frac{1}{q_2}} \esup_{x \in B(0, t)} v_1(x)^{-1}v_2(x)
\end{align}
and
\begin{align*}
I_{12} := \sup_{t \in (0,\infty)} \bigg( \int_0^t \bigg( \int_x^{\infty} w_1^{q_1} \bigg)^{-\frac{q_1}{q_1 -p_1}} w_1(x)^{q_1} 
\sup_{\tau \in B(x,t)} v_1(\tau)^{-p_1} v_2(\tau)^{p_1} dx \bigg)^{\frac{q_1-p_1}{q_1p_1}} \bigg( \int_t^{\infty} w_2^{q_2} \bigg)^{\frac{1}{q_2}}.
\end{align*}
Moreover, $\|\Id\|_{LM_1 \rightarrow LM_2} \approx  I_{11} + I_{12}$.

{\rm (ii)} If $ q_2 < q_1$, then $LM_1 \hookrightarrow LM_2$ for all $f \in \mathcal{M}^+(\mathbb{R}^n)$ if and only if $I_{11} < \infty$, $I_{13} < \infty$ and $I_{14} < \infty$, where $I_{11}$ is defined in \eqref{A_12},
\begin{align*}
I_{13} &:=  \left( \int_0^{\infty} \bigg( \int_0^t \bigg( \int_x^{\infty} w_1^{q_1} \bigg)^{-\frac{q_1}{q_1 - p_1}} w_1(x)^{q_1} dx \bigg)^{\frac{q_1(q_2-p_1)}{p_1(q_1 - q_2)}}  \bigg( \int_t^{\infty}  w_1^{q_1} \bigg)^{-\frac{q_1}{q_1 - p_1}}  w_1(t)^{q_1} \right. \\
& \hspace{2cm} \times \left. \sup_{x \in \dual B(0, t)} v_1(x)^{-\frac{q_2(q_1-p_1)}{q_1 -q_2}}  v_2(x)^{\frac{q_2(q_1-p_1)}{q_1 -q_2}} \bigg( \int_{|x|}^{\infty} w_2^{q_2} \bigg)^{\frac{q_1}{q_1 - q_2}}  dt \right)^{\frac{q_1-q_2}{q_1 q_2}}
\end{align*}
and
\begin{align*}
I_{14} &:=  \left( \int_0^{\infty} \bigg( \int_0^t \bigg( \int_x^{\infty} w_1^{q_1} \bigg)^{-\frac{q_1}{q_1 - p_1}} w_1(x)^{q_1} \sup_{z \in B(x,t)}  v_1(z)^{-p_1} v_2(z)^{p_1} dx \bigg)^{\frac{q_1(q_2-p_1)}{p_1(q_1 -q_2)}}  \right. \\
& \hspace{2cm} \times \left. \sup_{x \in \dual B(0,t)} v_1(x)^{-p_1} v_2(x)^{p_1} \bigg( \int_{|x|}^{\infty} w_2^{q_2} \bigg)^{\frac{q_1}{q_1 - q_2}} \bigg( \int_t^{\infty} w_1^{q_1} \bigg)^{-\frac{q_1}{q_1 - p_1}} w_1(t)^{q_1}   dt \right)^{\frac{q_1-q_2}{q_1 q_2}}.
\end{align*}
Moreover, $\|\Id\|_{LM_1 \rightarrow LM_2} \approx I_{11} + I_{13} + I_{14}$.
\end{thm}

\begin{proof}
As in the proof of Theorem~\ref{maintheorem2}, using Theorem~\ref{T:Morrey-Cesaro-p=1}, \eqref{main inequality equiv. p=1} holds. 

{\rm(i)} If $\theta_1 \leq \theta_2$, then $\theta \leq q$, applying  \cite[Theorem 1.4, (i)]{Unver}, we have that $\|\Id\|_{LM_1\rightarrow LM_2}^{p_1} \approx I + II$, where
\begin{equation*}
I  := \bigg(\int_0^{\infty} w \bigg)^{-\frac{1}{\theta}} \sup_{t \in (0,\infty)} \bigg( \int_t^{\infty} u \bigg)^{\frac{1}{q}} 
\esup_{s \in (0, t)} \tilde v(s)
\end{equation*}
and 
\begin{equation*}
II := \sup_{t \in (0,\infty)} \bigg( \int_0^t \bigg( \int_x^{\infty} w \bigg)^{-\frac{\theta}{\theta -1}} 
w(x) \sup_{z \in (x,t)} \tilde v(z) dx \bigg)^{\frac{\theta-1}{\theta}} \bigg( \int_t^{\infty} u\bigg)^{\frac{1}{q}}.
\end{equation*}
Using \eqref{1} and substituting the powers $p, q, \theta$, and weights $u, v, w$ into $I$ and $II$, we obtain that $\|\Id\|_{LM_1\rightarrow LM_2} \approx I_{11} + I_{12}$. 

{\rm (ii)} If $\theta_2 < \theta_1$, then $ q < \theta$, and applying  \cite[Theorem 1.4, 
(ii)]{Unver}, we have that 
$\|\Id\|_{LM_1\rightarrow LM_2}^{p_1} \approx I_{11}+ III + IV$, where
\begin{align*}
III &=  \bigg( \int_0^{\infty} \bigg( \int_0^t \bigg( \int_x^{\infty} w \bigg)^{-\frac{\theta}{\theta - 
1}} w(x) dx \bigg)^{\frac{\theta(q-1)}{\theta - q}}  \bigg( \int_t^{\infty} w \bigg)^{-\frac{\theta}{\theta - 1}} w(t) 
\bigg. \\ 
& \hspace{2cm} \times \bigg. \sup_{z \in (t,\infty)} \tilde v(z)^{\frac{q(\theta-1)}{\theta -q}} \bigg( \int_z^{\infty} 
u \bigg)^{\frac{\theta}{\theta - q}}  dt \bigg)^{\frac{\theta-q}{\theta q}}  
\end{align*}
and
\begin{align*}
IV &:=  \bigg( \int_0^{\infty} \bigg( \int_0^t \bigg( \int_x^{\infty} w \bigg)^{-\frac{\theta}{\theta - 1}} 
w(x) \sup_{z \in (x,t)} \tilde  v(z) dx \bigg)^{\frac{\theta(q-1)}{\theta - q}}  \bigg( \int_t^{\infty} w 
\bigg)^{-\frac{\theta}{\theta - 1}} w(t) \bigg. \\
& \hspace{2cm} \times \bigg. \sup_{z \in (t,\infty)} \tilde v(z) \bigg( \int_z^{\infty} u \bigg)^{\frac{\theta}{\theta 
- q}}  dt \bigg)^{\frac{\theta-q}{\theta q}}.
\end{align*}
Since, 
\begin{align*}
\sup_{z \in (t,\infty)} \tilde v(z)^{\frac{q(\theta-1)}{\theta -q}} \bigg( \int\limits_z^{\infty} 
u\bigg)^{\frac{\theta}{\theta - q}}  
&= \sup_{z \in (t,\infty)} \bigg(\esup_{|x|=z} v(x) \bigg)^{\frac{q(\theta-1)}{\theta -q}} \bigg( \int\limits_z^{\infty} u 
\bigg)^{\frac{\theta}{\theta - q}}\\
& = \sup_{z \in (t,\infty)} \esup_{|x|=z} v(x)^{\frac{q(\theta-1)}{\theta -q}} \bigg( \int\limits_{|x|}^{\infty} u 
\bigg)^{\frac{\theta}{\theta - q}}\\
& =\sup_{x \in \dual B(0, t)} v(x)^{\frac{q(\theta-1)}{\theta -q}} \bigg( \int\limits_{|x|}^{\infty} u 
\bigg)^{\frac{\theta}{\theta - q}}
\end{align*}
and, similarly
\begin{align*}
\sup_{z \in (t,\infty)} \tilde v(z) \bigg( \int\limits_z^{\infty} u \bigg)^{\frac{\theta}{\theta - q}} 
= \sup_{x \in \dual B(0, t)} v(x) \bigg( \int\limits_{|x|}^{\infty} u \bigg)^{\frac{\theta}{\theta - q}},
\end{align*}
substituting the powers $p, q, \theta$, and weights $u, v, w$ into $III$ and $IV$, the result follows. 
\end{proof}

\begin{rem}
We should mention that by the change of variables $x= y/|y|^2$ and $t = 1/\tau$, the embedding 
\begin{equation*}
\dual LM_{p_1, q_1}(v_1,w_1) \hookrightarrow \dual LM_{p_2, q_2}(v_2,w_2)
\end{equation*}
is equivalent to the embedding
\begin{equation*}
LM_{p_1, q_1}(\tilde v_1, \tilde w_1) \hookrightarrow  LM_{p_2, q_2}(\tilde v_2, \tilde w_2),
\end{equation*}
where $\tilde v_i(y) = v_i(y/|y|^2) |y|^{-2n/p_i}$ $\tilde w_i (\tau) = \tau^{-2/q_i} w_i(1/\tau)$. Therefore, by using 
the characterizations of \eqref{emb. LM1-LM2}, it is possible to give the characterizations of the embeddings between 
weighted complementary local Morrey-type spaces.
\end{rem}

\end{document}